\documentclass[10pt,a4paper]{article}
\usepackage[english]{babel}
\usepackage[margin=2.5cm]{geometry}
\usepackage{amsmath,amssymb,amsthm}
\usepackage{indentfirst}
\usepackage[hidelinks]{hyperref}

\newtheorem{theorem}{Theorem}[section]
\newtheorem{proposition}{Proposition}[section]
\newtheorem{lemma}{Lemma}[section]

\newcommand{\mbb}{\mathbb}

\title{A Central Limit Theorem for Coprime Pair Counts in Randomly Translated Disks}
\author{Xin Hang Ji\thanks{Corresponding author. Email:
		\texttt{1000535190@smail.shnu.edu.cn}}\\
	Department of Mathematics, Shanghai Normal University}
\date{}
\begin{document}
	\maketitle
	
	\begin{abstract}
		We study the distribution of the number of coprime pairs in randomly translated disks.
		We prove that its spatial variance is asymptotic to $\sigma^2r\log r$ and establish a central limit theorem under the corresponding normalization,
		where $\sigma^2$ is an explicit positive constant.
		By truncating the numerators of rational frequencies in lowest terms, the proof reduces higher-moment estimates to counting squarefree denominators in zero-sum relations.
	\end{abstract}
	
	\section{Introduction}
	
	The limiting distribution of lattice-point counts depends on the averaging procedure and the shape of the region.
	Hughes--Rudnick~\cite{HR2004} studied integer lattice-point counts in thin annuli under averaging over the radius
	and proved a central limit theorem when the width of the annulus tends to zero sufficiently slowly.
	
	Sugita--Takanobu~\cite{ST2003} studied coprime pair counts in integer translates of squares,
	showing that the normalized limits depend on the arithmetic properties of the side length.
	Fern\'andez--Fern\'andez~\cite{FF2025} subsequently obtained explicit formulas for the spatial limiting distribution of the count for a fixed side length.
	
	In~\cite{Ji2026}, the author studied integer lattice-point counts in randomly translated high-dimensional balls,
	obtaining normal and lognormal limits under different growth conditions on the radius.
	The present paper concerns coprime pair counts in randomly translated two-dimensional disks.
	
	Let $|\cdot|$ denote the Euclidean norm on $\mathbb R^2$, let $\operatorname{meas}$ denote two-dimensional Lebesgue measure,
	and set $e(t)=\exp(2\pi i t)$. Define the set of coprime pairs and its counting function in a disk by
	\[
	\begin{aligned}
		\mathbb V
		&=\{(n_1,n_2)\in\mathbb Z^2:\gcd(n_1,n_2)=1\},\\
		G(r;\mathbf x)
		&=\#\{\mathbf n\in\mathbb V:|\mathbf x-\mathbf n|\le r\}
		\qquad(r>0,\ \mathbf x\in\mathbb R^2).
	\end{aligned}
	\tag*{(1.1)}
	\]
	
	We take spatial averages of the center over expanding disks: for each fixed $r$,
	we first let $\mathbf x$ be uniformly distributed in $|\mathbf x|\le R$ and take $R\to\infty$, and then let $r\to\infty$.
	With this order of limits, we prove that the spatial variance of the count is asymptotic to $\sigma^2r\log r$ (where $\sigma^2$ is a positive constant)
	and obtain a normal limit on the scale $\sqrt{\sigma^2r\log r}$.
	
	\begin{theorem}\label{thm:clt}
		Let
		\[
		\sigma^2
		=\frac{18}{\pi^6}
		\left(\sum_{\mathbf m\in\mathbb Z^2\setminus\{\mathbf0\}}
		|\mathbf m|^{-3}\right)
		\prod_p\left(1-\frac1{(p+1)^2}\right).
		\tag*{(1.2)}
		\]
		Then, for every $\nu\in\mathbb R$,
		\[
		\begin{aligned}
			&\lim_{r\to\infty}\lim_{R\to\infty}
			\frac1{\pi R^2}\operatorname{meas}\left\{
			\mathbf x\in\mathbb R^2:|\mathbf x|\le R,\quad
			\frac{G(r;\mathbf x)-6r^2/\pi}
			{\sqrt{\sigma^2r\log r}}\ge\nu\right\}\\
			&\qquad=\frac1{\sqrt{2\pi}}\int_\nu^\infty e^{-y^2/2}\,dy.
		\end{aligned}
		\tag*{(1.3)}
		\]
	\end{theorem}
	
	The proof uses a truncation of the numerators of rational frequencies, so that the truncated moment expansions are absolutely convergent series.
	The fully paired terms yield the Gaussian moments, while the unpaired terms are controlled by counting squarefree denominators in zero-sum relations.
	Finally, a second-moment error estimate is used to remove the truncation.
	The denominator-counting method follows Konyagin--Korolev~\cite{KK2016};
	the sums arising in the variance calculation are handled using the higher derivative estimates of Heath-Brown~\cite{HB2016}.
	
	Throughout, $p$ denotes a prime, $\mu$ denotes the M\"obius function,
	$\log$ denotes the natural logarithm, and $\mathbf1_A$ denotes the indicator function of a set $A$.
	The constants implicit in $O$ and $\ll$ without subscripts are absolute; subscripts indicate the parameters on which these constants may depend.
	
	\section{Notation and Propositions}
	
	For $r>0$ and an integer $k\ge1$, let
	\[
	X_r(\mathbf x)=\frac{G(r;\mathbf x)}r-\frac{6r}{\pi},
	\qquad
	\mathcal M_k(r,R)=\frac1{\pi R^2}
	\int_{|\mathbf x|\le R}X_r(\mathbf x)^k\,d\mathbf x.
	\tag*{(2.1)}
	\]
	For a function $f:\mbb R^2\to\mbb C$, we denote its mean by
	\[
	\langle f\rangle
	=\lim_{R\to\infty}\frac1{\pi R^2}
	\int_{|\mathbf x|\le R}f(\mathbf x)\,d\mathbf x.
	\tag*{(2.2)}
	\]
	
	For $\xi\in\mathbb Q^2\setminus\{\mathbf0\}$, let $d(\xi)$ be the least positive integer such that
	$d(\xi)\xi\in\mathbb Z^2$.
	Then $\xi$ has a unique representation as $\mathbf m/q$, where
	$q=d(\xi)$, $\mathbf m\in\mathbb Z^2\setminus\{\mathbf0\}$, and
	$\gcd(m_1,m_2,q)=1$.
	Set $H_r(0)=0$ and define
	\[
	H(\xi)=H_r(\xi)
	=\frac6{\pi^2}
	\frac{\mu(d(\xi))}{\displaystyle\prod_{p\mid d(\xi)}(p^2-1)}
	\frac{J_1(2\pi r|\xi|)}{|\xi|}
	\qquad(\xi\ne\mathbf0),
	\tag*{(2.3)}
	\]
	where $J_1$ is the Bessel function of the first kind of order one.
	It is immediate that $H(-\xi)=H(\xi)\in\mathbb R$; if $d(\xi)$ is not squarefree, then $H(\xi)=0$.
	
	For $z\ge2$ and an integer $N\ge1$, let
	\[
	Q_z=\prod_{p\le z}p,
	\qquad
	w_N(\mathbf m)=\prod_{1\le\ell\le2}
	\left(1-\frac{|m_\ell|}{N+1}\right)_+,
	\tag*{(2.4)}
	\]
	where $t_+=\max\{t,0\}$, and define the finite sum
	\[
	\mathcal C_k(z,N)
	=\sum_{\substack{
			\mathbf m_1,\ldots,\mathbf m_k\in\mathbb Z^2\\
			\|\mathbf m_j\|_\infty\le N\ (1\le j\le k)\\
			\mathbf m_1+\cdots+\mathbf m_k=\mathbf0}}
	\prod_{1\le j\le k}
	\left[w_N(\mathbf m_j)
	H\!\left(\frac{\mathbf m_j}{Q_z}\right)\right].
	\tag*{(2.5)}
	\]
	
	For $r\ge3$ and $M\ge1$, define the numerator truncation
	\[
	X_{r,M}(\mathbf x)
	=\sum_{\substack{q\ge1\\q\text{ squarefree}}}
	\sum_{\substack{\mathbf m\in\mathbb Z^2\\0<|\mathbf m|\le M\\
			\gcd(m_1,m_2,q)=1}}
	H\!\left(\frac{\mathbf m}{q}\right)
	e\!\left(-\frac{\mathbf m\cdot\mathbf x}{q}\right).
	\tag*{(2.6)}
	\]
	The convergence of the series (2.6) is established in Lemma~\ref{lem:tail}.
	
	\begin{proposition}\label{prop:fourier}
		For every $r>0$ and every integer $k\ge1$, the limit $\lim_{R\to\infty}\mathcal M_k(r,R)$ exists, and
		\[
		\langle X_r^k\rangle
		=\lim_{z\to\infty}\lim_{N\to\infty}\mathcal C_k(z,N).
		\tag*{(2.7)}
		\]
		In particular,
		\[
		\langle X_r\rangle=0,
		\qquad
		\langle X_r^2\rangle=\sum_{\xi\in\mathbb Q^2\setminus\{\mathbf0\}}|H(\xi)|^2<\infty,
		\tag*{(2.8)}
		\]
		The series for the second moment is absolutely convergent in the usual sense.
	\end{proposition}
	
	\begin{proposition}\label{prop:moments}
		Suppose that $r\ge3$ and $1\le M\le\log r$.
		For every fixed integer $k\ge3$,
		\[
		\langle X_{r,M}^k\rangle
		=\begin{cases}
			\displaystyle
			\frac{k!}{2^{k/2}(k/2)!}\langle X_{r,M}^2\rangle^{k/2}
			+O_k\!\left(\frac{M^2(\log r)^{k/2-2}}{r^{k/2}}\right),
			& k\text{ is even},\\[6pt]
			\displaystyle
			O_k\!\left(\frac{M^{3/2}(\log r)^{(k-3)/2}}{r^{k/2}}\right),
			& k\text{ is odd}.
		\end{cases}
		\tag*{(2.9)}
		\]
		Moreover, $\langle X_{r,M}\rangle=0$.
	\end{proposition}
	
	\begin{proposition}\label{prop:variance}
		As $r\to\infty$,
		\[
		\langle X_r^2\rangle=\sigma^2\frac{\log r}r
		+o\!\left(\frac{\log r}r\right),
		\tag*{(2.10)}
		\]
		where $\sigma^2$ is defined in (1.2).
	\end{proposition}
	
	\section{Proof of Proposition 2.1}
	
	\begin{proof}
		Write
		\[
		F(r;\mathbf x)
		=\#\{\mathbf n\in\mathbb Z^2:|\mathbf x-\mathbf n|\le r\}.
		\tag*{(3.1)}
		\]
		The Poisson summation formula gives
		\[
		F(r;\mathbf x)
		=\pi r^2+
		\sum_{\mathbf m\in\mathbb Z^2\setminus\{\mathbf0\}}
		\frac r{|\mathbf m|}J_1(2\pi r|\mathbf m|)
		e(-\mathbf m\cdot\mathbf x),
		\tag*{(3.2)}
		\]
		with equality in $L^2(\mathbb R^2/\mathbb Z^2)$.
		For a fixed $\xi\in\mathbb Q^2\setminus\{\mathbf0\}$, the condition
		$\mathbf m/q=\xi$ is equivalent to $d(\xi)\mid q$ and $\mathbf m=q\xi$.
		Thus, the Euler product gives the absolutely convergent identity
		\[
		\begin{aligned}
			&\sum_{\substack{q\ge1,\ \mathbf m\in\mathbb Z^2\setminus\{\mathbf0\}\\
					\mathbf m/q=\xi}}
			\frac{\mu(q)}{q|\mathbf m|}
			J_1\!\left(\frac{2\pi r|\mathbf m|}{q}\right)\\
			&\qquad=\frac{J_1(2\pi r|\xi|)}{|\xi|}
			\sum_{\substack{q\ge1\\d(\xi)\mid q}}\frac{\mu(q)}{q^2}
			=H(\xi).
		\end{aligned}
		\tag*{(3.3)}
		\]
		
		Define
		\[
		\begin{aligned}
			G_z(r;\mathbf x)
			&=
			\#\left\{
			\mathbf n\in\mathbb Z^2:
			|\mathbf x-\mathbf n|\le r,\,
			\gcd(n_1,n_2,Q_z)=1
			\right\}\\
			&=
			\sum_{q\mid Q_z}
			\mu(q)
			F\left(\frac rq;\frac{\mathbf x}{q}\right),
		\end{aligned}
		\tag*{(3.4)}
		\]
		and let
		\[
		c_z=\prod_{p>z}\left(1-\frac1{p^2}\right),
		\qquad
		\widetilde X_{r,z}(\mathbf x)
		=
		\frac{c_z}{r}G_z(r;\mathbf x)-\frac{6r}{\pi}.
		\tag*{(3.5)}
		\]
		The function $\widetilde X_{r,z}$ is $Q_z\mathbb Z^2$-periodic.
		Since
		\[
		0\le G(r;\mathbf x)\le G_z(r;\mathbf x)
		\le F(r;\mathbf x)\ll_r1,
		\tag*{(3.6)}
		\]
		there is a constant $C_r$, depending only on $r$, such that
		\[
		|X_r(\mathbf x)|\le C_r,
		\qquad
		|\widetilde X_{r,z}(\mathbf x)|\le C_r
		\tag*{(3.7)}
		\]
		for all $\mathbf x\in\mathbb R^2$ and $z\ge2$.
		
		By (3.2), (3.4), and
		$c_z\prod_{p\le z}(1-p^{-2})=6/\pi^2$, 
		the constant term of $\widetilde X_{r,z}$ is zero.
		For $\xi\in Q_z^{-1}\mathbb Z^2\setminus\{\mathbf0\}$,
		we have $d(\xi)\mid Q_z$, and hence
		\[
		\begin{aligned}
			&\frac1{Q_z^2}
			\int_{[0,Q_z]^2}
			\widetilde X_{r,z}(\mathbf x)e(\xi\cdot\mathbf x)\,d\mathbf x\\
			&\quad=
			c_z\frac{J_1(2\pi r|\xi|)}{|\xi|}
			\sum_{\substack{q\mid Q_z\\d(\xi)\mid q}}
			\frac{\mu(q)}{q^2}\\
			&\quad=
			c_z\frac{J_1(2\pi r|\xi|)}{|\xi|}
			\frac{\mu(d(\xi))}{d(\xi)^2}
			\prod_{\substack{p\le z\\p\nmid d(\xi)}}
			\left(1-\frac1{p^2}\right)\\
			&\quad=H(\xi).
		\end{aligned}
		\tag*{(3.8)}
		\]
		
		Let
		\[
		P_{z,N}(\mathbf x)
		=
		\sum_{\substack{
				\mathbf m\in\mathbb Z^2\\
				\|\mathbf m\|_\infty\le N
		}}
		w_N(\mathbf m)
		H\left(\frac{\mathbf m}{Q_z}\right)
		e\left(-\frac{\mathbf m\cdot\mathbf x}{Q_z}\right).
		\tag*{(3.9)}
		\]
		The polynomial $P_{z,N}$ is the two-dimensional product Fej\'er mean of $\widetilde X_{r,z}$.
		For fixed $z$,
		\[
		\|P_{z,N}\|_\infty\le C_r,
		\qquad
		\int_{[0,Q_z]^2}
		|P_{z,N}(\mathbf x)-\widetilde X_{r,z}(\mathbf x)|
		\,d\mathbf x
		\longrightarrow0
		\quad(N\to\infty).
		\tag*{(3.10)}
		\]
		By uniform boundedness, $P_{z,N}^k\to \widetilde X_{r,z}^k$
		in $L^1([0,Q_z]^2)$.
		Therefore, by orthogonality in finite Fourier sums,
		\[
		\begin{aligned}
			m_{z,k}
			&:=
			\frac1{Q_z^2}
			\int_{[0,Q_z]^2}
			\widetilde X_{r,z}(\mathbf x)^k\,d\mathbf x\\
			&=
			\lim_{N\to\infty}
			\frac1{Q_z^2}
			\int_{[0,Q_z]^2}
			P_{z,N}(\mathbf x)^k\,d\mathbf x\\
			&=
			\lim_{N\to\infty}\mathcal C_k(z,N).
		\end{aligned}
		\tag*{(3.11)}
		\]
		
		Write $\varepsilon_z=\sum_{p>z}p^{-2}$.
		Then $0\le1-c_z\le\varepsilon_z$ and $\varepsilon_z\to0$.
		By the definition of $G_z$,
		\[
		0\le G_z(r;\mathbf x)-G(r;\mathbf x)
		\le
		\sum_{p>z}
		\sum_{\mathbf n\in p\mathbb Z^2\setminus\{\mathbf0\}}
		\mathbf1_{\{|\mathbf x-\mathbf n|\le r\}}.
		\tag*{(3.12)}
		\]
		For every $S>0$,
		\[
		\#\left\{
		\mathbf n\in p\mathbb Z^2\setminus\{\mathbf0\}:
		|\mathbf n|\le S
		\right\}
		\ll\frac{S^2}{p^2}.
		\tag*{(3.13)}
		\]
		Consequently, for $R\ge1$,
		\[
		\begin{aligned}
			&\frac1{\pi R^2}
			\int_{|\mathbf x|\le R}
			\bigl(G_z(r;\mathbf x)-G(r;\mathbf x)\bigr)
			\,d\mathbf x\\
			&\quad\le
			\frac{r^2}{R^2}
			\sum_{p>z}
			\#\left\{
			\mathbf n\in p\mathbb Z^2\setminus\{\mathbf0\}:
			|\mathbf n|\le R+r
			\right\}\\
			&\quad\ll_r\varepsilon_z.
		\end{aligned}
		\tag*{(3.14)}
		\]
		On the other hand,
		\[
		|X_r(\mathbf x)-\widetilde X_{r,z}(\mathbf x)|
		\le
		\frac{G_z(r;\mathbf x)-G(r;\mathbf x)}r
		+
		\frac{1-c_z}{r}G_z(r;\mathbf x).
		\tag*{(3.15)}
		\]
		Combining (3.7), (3.14), and
		$|a^k-b^k|\le kC_r^{k-1}|a-b|$, we obtain
		\[
		\sup_{R\ge1}
		\frac1{\pi R^2}
		\int_{|\mathbf x|\le R}
		\left|
		X_r(\mathbf x)^k-\widetilde X_{r,z}(\mathbf x)^k
		\right|
		\,d\mathbf x
		\ll_{r,k}\varepsilon_z.
		\tag*{(3.16)}
		\]
		
		If $f$ is a bounded measurable $Q\mathbb Z^2$-periodic function,
		we decompose the disk into complete fundamental cells and a boundary portion.
		The latter has area $O(QR+Q^2)$, so
		\[
		\begin{aligned}
			\frac1{\pi R^2}
			\int_{|\mathbf x|\le R}f(\mathbf x)\,d\mathbf x
			&=
			\frac1{Q^2}
			\int_{[0,Q]^2}f(\mathbf x)\,d\mathbf x\\
			&\quad+
			O\left(
			\|f\|_\infty
			\left(\frac QR+\frac{Q^2}{R^2}\right)
			\right).
		\end{aligned}
		\tag*{(3.17)}
		\]
		For fixed $z$, applying (3.17) to $f=\widetilde X_{r,z}^k$
		and then using (3.16), we obtain
		\[
		\limsup_{R\to\infty}
		\left|
		\mathcal M_k(r,R)-m_{z,k}
		\right|
		\ll_{r,k}\varepsilon_z.
		\tag*{(3.18)}
		\]
		Thus, for any $z,z'\ge2$,
		\[
		|m_{z,k}-m_{z',k}|
		\ll_{r,k}\varepsilon_z+\varepsilon_{z'}.
		\tag*{(3.19)}
		\]
		Hence $m_{z,k}$ converges as $z\to\infty$.
		Using (3.18) again and substituting (3.11), we obtain
		\[
		\begin{aligned}
			\lim_{R\to\infty}\mathcal M_k(r,R)
			&=
			\lim_{z\to\infty}m_{z,k}\\
			&=
			\lim_{z\to\infty}\lim_{N\to\infty}
			\mathcal C_k(z,N).
		\end{aligned}
		\tag*{(3.20)}
		\]
		
		Since the constant term of each $\widetilde X_{r,z}$ is zero,
		\[
		\lim_{R\to\infty}\mathcal M_1(r,R)=0.
		\tag*{(3.21)}
		\]
		For the second moment, Parseval's identity gives
		\[
		m_{z,2}
		=
		\sum_{\mathbf m\in\mathbb Z^2\setminus\{\mathbf0\}}
		\left|H\left(\frac{\mathbf m}{Q_z}\right)\right|^2
		\le C_r^2.
		\tag*{(3.22)}
		\]
		The sets $Q_z^{-1}\mathbb Z^2$ increase with $z$,
		and their union contains all rational frequencies with squarefree denominators,
		while $H$ vanishes at all other frequencies.
		By the monotone convergence theorem,
		\[
		\begin{aligned}
			\lim_{R\to\infty}\mathcal M_2(r,R)
			&=
			\lim_{z\to\infty}m_{z,2}\\
			&=
			\sum_{\xi\in\mathbb Q^2\setminus\{\mathbf0\}}
			|H(\xi)|^2\\
			&=
			\sum_{\xi\in\mathbb Q^2}
			H(\xi)H(-\xi)
			<\infty.
		\end{aligned}
		\tag*{(3.23)}
		\]
	\end{proof}
	
	\section{Proof of Proposition 2.2}
	
	\begin{lemma}\label{lem:tail}
		For $r\ge3$ and $M\ge1$, the series (2.6) converges absolutely and uniformly,
		the function $X_{r,M}$ is real-valued, and
		\[
		\begin{aligned}
			\langle X_r^2\rangle&\ll\frac{\log r}{r},\\
			\langle|X_r-X_{r,M}|^2\rangle
			&=\langle X_r^2\rangle-\langle X_{r,M}^2\rangle\\
			&\ll\frac{\log r+\log(2M)}{rM}.
		\end{aligned}
		\tag*{(4.1)}
		\]
	\end{lemma}
	
	\begin{proof}
		
		Let $\mathbf m\in\mathbb Z^2\setminus\{\mathbf0\}$ and suppose that $\gcd(m_1,m_2,q)=1$. Using $|J_1(t)|\ll\min(t,t^{-1/2})$ and
		$\prod_{p\mid q}(1-p^{-2})^{-1}\le\zeta(2)$, we obtain
		\[
		\left|H\!\left(\frac{\mathbf m}{q}\right)\right|
		\ll
		\min\left\{\frac r{q^2},
		\frac1{\sqrt{rq}\,|\mathbf m|^{3/2}}\right\}.
		\tag*{(4.2)}
		\]
		For fixed $\mathbf m\ne\mathbf0$, splitting the sum at $q=r|\mathbf m|$ gives
		\[
		\begin{aligned}
			\sum_{\substack{q\ge1\\\gcd(m_1,m_2,q)=1}}\left|H\!\left(\frac{\mathbf m}{q}\right)\right|
			&\ll |\mathbf m|^{-1},\\
			\sum_{\substack{q\ge1\\\gcd(m_1,m_2,q)=1}}\left|H\!\left(\frac{\mathbf m}{q}\right)\right|^2
			&\ll\frac{1+\log(r|\mathbf m|)}{r|\mathbf m|^3}.
		\end{aligned}
		\tag*{(4.3)}
		\]
		When summing the bounds in (4.2), we may drop the restrictions that the fractions be in lowest terms and that the denominators be squarefree.
		The first estimate, together with $H(-\xi)=H(\xi)\in\mathbb R$, shows that the series (2.6) converges absolutely and uniformly and is real-valued.
		The second estimate, together with
		\[
		\sum_{|\mathbf m|>M}
		\frac{1+\log(r|\mathbf m|)}{|\mathbf m|^3}
		\ll\frac{\log r+\log(2M)}M
		\tag*{(4.4)}
		\]
		gives the bound for the tail of the sum of squares.
		
		By (3.8), the spatial mean property of periodic functions, and the estimate (3.16) with $k=1$, we have
		\[
		\langle X_r e(\xi\cdot\mathbf x)\rangle=H(\xi)
		\qquad(\xi\in\mathbb Q^2\setminus\{\mathbf0\}).
		\tag*{(4.4a)}
		\]
		Here, if $d(\xi)$ is squarefree, we first take $z$ sufficiently large that
		$d(\xi)\mid Q_z$. Equation (3.8) then gives the corresponding periodic mean, after which we let $z\to\infty$.
		If $d(\xi)$ is not squarefree, then $\xi\notin Q_z^{-1}\mathbb Z^2$
		for every $z$, so $\langle\widetilde X_{r,z}e(\xi\cdot\mathbf x)\rangle=0$;
		letting $z\to\infty$ in (3.16) shows that the left-hand side of (4.4a) is
		$0=H(\xi)$.
		
		Write
		\[
		\mathcal A_M
		=\left\{\xi\in\mathbb Q^2\setminus\{\mathbf0\}:
		|d(\xi)\xi|\le M\right\}.
		\]
		Take a finite symmetric set $A\subseteq\mathcal A_M$ and let
		\[
		P_A(\mathbf x)=\sum_{\xi\in A}H(\xi)e(-\xi\cdot\mathbf x).
		\]
		By (4.4a) and orthogonality in finite Fourier sums,
		\[
		\begin{aligned}
			\langle P_A^2\rangle
			&=\sum_{\xi\in A}H(\xi)H(-\xi)
			=\sum_{\xi\in A}|H(\xi)|^2,\\
			\langle X_rP_A\rangle
			&=\sum_{\xi\in A}H(\xi)
			\langle X_re(-\xi\cdot\mathbf x)\rangle
			=\sum_{\xi\in A}|H(\xi)|^2.
		\end{aligned}
		\tag*{(4.4b)}
		\]
		Consequently,
		\[
		\langle|X_r-P_A|^2\rangle
		=\langle X_r^2\rangle-\sum_{\xi\in A}|H(\xi)|^2.
		\tag*{(4.4c)}
		\]
		Let the finite symmetric sets $A$ increase to exhaust $\mathcal A_M$. By the absolute and uniform convergence established above,
		$P_A\to X_{r,M}$ uniformly; since $X_r$ is bounded,
		we may pass to the limit in each term of (4.4b)--(4.4c). Combining this with (3.23), we obtain
		\[
		\begin{aligned}
			\langle X_{r,M}^2\rangle
			&=\sum_{\substack{\xi\in\mathbb Q^2\setminus\{\mathbf0\}\\
					|d(\xi)\xi|\le M}}|H(\xi)|^2,\\
			\langle|X_r-X_{r,M}|^2\rangle
			&=\sum_{\substack{\xi\in\mathbb Q^2\setminus\{\mathbf0\}\\
					|d(\xi)\xi|>M}}|H(\xi)|^2.
		\end{aligned}
		\tag*{(4.5)}
		\]
		Summing over all $\mathbf m\ne\mathbf0$ also yields $\langle X_r^2\rangle\ll r^{-1}\log r$.
		
	\end{proof}
	
	\begin{lemma}\label{lem:count}
		Fix an integer $s\ge3$ and nonzero vectors
		$\mathbf m_1,\ldots,\mathbf m_s\in\mathbb Z^2$.
		Let $\mathcal U$ be the set of positive integer tuples
		$(q_1,\ldots,q_s)$ satisfying the following conditions: each $q_i$ is squarefree,
		$\gcd(m_{i,1},m_{i,2},q_i)=1$, and
		\[
		\sum_{1\le i\le s}\frac{\mathbf m_i}{q_i}=\mathbf0,
		\qquad
		\frac{\mathbf m_i}{q_i}+\frac{\mathbf m_j}{q_j}\ne\mathbf0
		\quad(i\ne j).
		\tag*{(4.6)}
		\]
		Then
		\[
		\sum_{(q_1,\ldots,q_s)\in\mathcal U}
		\frac1{\sqrt{q_1\cdots q_s}}\ll_s1.
		\tag*{(4.7)}
		\]
		The bound is uniform over all nonzero integer vectors $\mathbf m_i$.
	\end{lemma}
	
	\begin{proof}
		
		Let $P=q_1\cdots q_s$ and $c=1/(s(s-1))$, and write
		$N_{\mathbf m}(T)=\#\{\mathbf q\in\mathcal U:P\le T\}$.
		We first prove that, for every $\varepsilon>0$,
		\[
		N_{\mathbf m}(T)
		\ll_{s,\varepsilon}
		T^{1/2-c/2+\varepsilon}
		\qquad(T\ge2).
		\tag*{(4.8)}
		\]
		
		Every prime dividing $P$ divides at least two of the $q_i$.
		Indeed, if $p$ divided only $q_i$, multiplying (4.6) by
		$\operatorname{lcm}(q_1,\ldots,q_s)$ and reducing modulo $p$
		would give $p\mid m_{i,1}$ and $p\mid m_{i,2}$, contradicting the lowest-terms condition.
		Thus every prime factor of $P$ has exponent at least $2$, and there is a unique factorization
		\[
		P=u^3v^2,
		\qquad u\text{ is squarefree}.
		\tag*{(4.9)}
		\]
		Here $u$ is the product of the prime factors whose exponents in $P$ are odd, so $P/u$ is a perfect square.
		Let $\tau_s(n)$ denote the number of ordered factorizations of $n$ into a product of $s$ positive integers. We use the standard estimate $\tau_s(n)\ll_{s,\varepsilon}n^\varepsilon$ below.
		
		For tuples with $u>T^c$, there are at most $\tau_s(P)$ ordered factorizations for each fixed $P$, so their number is at most
		\[
		\begin{aligned}
			\sum_{u>T^c}\sum_{v\le T^{1/2}u^{-3/2}}\tau_s(u^3v^2)
			&\ll_{s,\varepsilon}
			T^{1/2+\varepsilon}\sum_{u>T^c}u^{-3/2}\\
			&\ll_{s,\varepsilon}T^{1/2-c/2+\varepsilon}.
		\end{aligned}
		\tag*{(4.10)}
		\]
		
		When $u\le T^c$, choose $q_i=\max_jq_j$.
		The integer $q_i$ is squarefree, and each of its prime factors also divides some other $q_j$, so
		\[
		q_i\mid\prod_{j\ne i}\gcd(q_i,q_j).
		\tag*{(4.11)}
		\]
		We can therefore choose $j\ne i$ such that $d=\gcd(q_i,q_j)$ satisfies
		\[
		d\ge q_i^{1/(s-1)}\ge P^c.
		\tag*{(4.12)}
		\]
		Write $q_i=da$ and $q_j=db$. Equation (4.6) becomes
		\[
		\frac1d\left(\frac{\mathbf m_i}{a}
		+\frac{\mathbf m_j}{b}\right)
		+\sum_{\ell\ne i,j}\frac{\mathbf m_\ell}{q_\ell}
		=\mathbf0.
		\tag*{(4.13)}
		\]
		The vector in parentheses is nonzero, since otherwise the original tuple would contain a pair of opposite frequencies.
		Hence, once $a,b$ and the remaining $q_\ell$ are fixed, at most one value of $d$ satisfies this equation.
		
		Since $d$ is squarefree and $d^2\mid P$, writing $P$ in the form
		$P=u n^2$ gives $d\mid n$. Set $w=n/d$. Then
		\[
		ab\prod_{\ell\ne i,j}q_\ell
		=\frac P{d^2}=uw^2\le T^{1-2c},
		\qquad
		w\le T^{1/2-c}u^{-1/2}.
		\tag*{(4.14)}
		\]
		Taking into account the at most $s(s-1)$ choices of $i,j$, the number of tuples in this class is at most
		\[
		\begin{aligned}
			s(s-1)\sum_{u\le T^c}
			\sum_{w\le T^{1/2-c}u^{-1/2}}\tau_s(uw^2)
			&\ll_{s,\varepsilon}
			T^{1/2-c+\varepsilon}\sum_{u\le T^c}u^{-1/2}\\
			&\ll_{s,\varepsilon}T^{1/2-c/2+\varepsilon}.
		\end{aligned}
		\tag*{(4.15)}
		\]
		This proves (4.8), with a constant independent of $\mathbf m_1,\ldots,\mathbf m_s$.
		Taking $\varepsilon=c/4$ and splitting into dyadic ranges of $P$, we obtain
		\[
		\sum_{\mathbf q\in\mathcal U}P^{-1/2}
		\ll_s1+\sum_{h\ge0}2^{-h/2}
		(2^{h+1})^{1/2-c/4}
		\ll_s1.
		\tag*{(4.16)}
		\]
		
	\end{proof}
	
	\begin{proof}[Proof of Proposition~\ref{prop:moments}]
		
		By Lemma~\ref{lem:tail}, taking spatial means term by term gives
		\[
		\langle X_{r,M}^k\rangle
		=\sum_{\substack{\xi_1+\cdots+\xi_k=0\\
				0<|d(\xi_i)\xi_i|\le M\ (1\le i\le k)}}
		H(\xi_1)\cdots H(\xi_k).
		\tag*{(4.17)}
		\]
		This series is absolutely convergent.
		
		Let $U_s(r,M)$ denote the sum of the absolute values of the terms in (4.17) with $s$ frequencies
		such that no two frequencies are opposites of one another.
		By (4.2) and Lemma~\ref{lem:count},
		\[
		\begin{aligned}
			U_s(r,M)
			&\ll_s r^{-s/2}
			\sum_{0<|\mathbf m_1|,\ldots,|\mathbf m_s|\le M}
			\prod_{1\le i\le s}|\mathbf m_i|^{-3/2}\\
			&\ll_s r^{-s/2}M^{s/2}
			\qquad(s\ge3).
		\end{aligned}
		\tag*{(4.18)}
		\]
		
		Starting from any zero-sum tuple of frequencies, successively remove pairs of opposite frequencies until no further removal is possible.
		To make the choice unambiguous, fix a lexicographic order on pairs of positions in advance and, at each step, remove
		the lexicographically smallest removable pair. If $s$ frequencies remain, they still sum to zero and contain no pair of opposites.
		A nonempty remainder must have $s\ge3$; when $k$ is even, parity forces $s\ge4$.
		
		Let $\mathcal R_k(r,M)$ denote the contribution to (4.17) from zero-sum tuples that cannot be fully paired.
		Fix the $(k-s)/2$ disjoint pairs of positions that have been removed and the remaining $s$ positions.
		When estimating absolute values, dropping the additional restrictions imposed by the lexicographic rule only enlarges the summation range.
		The frequencies at each removed pair of positions can independently be written as $(\xi,-\xi)$, and the sum of their absolute contributions is
		$\langle X_{r,M}^2\rangle$; the sum of the absolute contributions from the remainder is at most $U_s(r,M)$.
		There are only $O_k(1)$ such decompositions of the positions, so
		\[
		|\mathcal R_k(r,M)|\ll_k
		\sum_{\substack{3\le s\le k\\s\equiv k\pmod2}}
		\langle X_{r,M}^2\rangle^{(k-s)/2}U_s(r,M).
		\tag*{(4.19)}
		\]
		By $\langle X_{r,M}^2\rangle\ll r^{-1}\log r$ and (4.18),
		the term corresponding to $s$ in (4.19) is at most
		\[
		r^{-k/2}(\log r)^{(k-s)/2}M^{s/2}.
		\tag*{(4.19a)}
		\]
		Since $M\le\log r$, when $k$ is odd, the bound in (4.19a) is dominated by the bound for $s=3$;
		when $k$ is even, the bound for $s=4$ applies. This gives the respective error bounds
		$M^{3/2}(\log r)^{(k-3)/2}r^{-k/2}$ and
		$M^2(\log r)^{k/2-2}r^{-k/2}$ in (2.9).
		
		We now treat the fully paired contribution to the even moments. Write $k=2j$ and let $\mathfrak P_{2j}$
		be the set of all perfect pairings of $\{1,\ldots,2j\}$.
		For each $\mathcal P\in\mathfrak P_{2j}$, require the frequencies at each pair of positions to be opposites of one another.
		The sum of the corresponding terms is then exactly $\langle X_{r,M}^2\rangle^j$. Summing over all pairings therefore gives
		\[
		\frac{(2j)!}{2^j j!}\langle X_{r,M}^2\rangle^j,
		\tag*{(4.19b)}
		\]
		However, a frequency tuple satisfying more than one pairing is counted repeatedly in (4.19b).
		
		For a fully paired tuple $\boldsymbol\xi=(\xi_1,\ldots,\xi_{2j})$,
		let $n(\boldsymbol\xi)$ be the number of perfect pairings that it satisfies. If $\mathcal F_{2j}$
		denotes the total contribution of fully paired tuples to (4.17), with each tuple counted once, then the difference between (4.19b)
		and $\mathcal F_{2j}$ is
		\[
		\sum_{\substack{\boldsymbol\xi\in\mathcal A_M^{2j}\\
				n(\boldsymbol\xi)\ge1}}
		\bigl(n(\boldsymbol\xi)-1\bigr)
		\prod_{1\le i\le2j}H(\xi_i).
		\]
		Since $n-1\le n(n-1)$, its absolute value is at most
		\[
		\sum_{\substack{\mathcal P,\mathcal P'\in\mathfrak P_{2j}\\
				\mathcal P\ne\mathcal P'}}
		\sum_{\substack{\boldsymbol\xi\in\mathcal A_M^{2j}\\
				\boldsymbol\xi\text{ satisfies both }\mathcal P,\mathcal P'}}
		\prod_{1\le i\le2j}|H(\xi_i)|.
		\tag*{(4.19c)}
		\]
		Superimposing two distinct pairings $\mathcal P,\mathcal P'$ produces a multigraph that decomposes into even cycles.
		An edge common to both pairings gives a cycle of length $2$; since $\mathcal P\ne\mathcal P'$,
		at least one cycle has length $2\ell\ge4$. Within a cycle of length $2\ell$,
		the pairing conditions force the frequencies to alternate between $\xi$ and $-\xi$. For $\ell\ge2$,
		using (4.2) and enlarging the summation range gives
		\[
		\begin{aligned}
			\sum_{\substack{\xi\in\mathbb Q^2\setminus\{\mathbf0\}\\
					|d(\xi)\xi|\le M}} |H(\xi)|^{2\ell}
			&\le \sum_{\xi\in\mathbb Q^2\setminus\{\mathbf0\}}|H(\xi)|^{2\ell}\\
			&\ll_\ell r^{-\ell}
			\sum_{q\ge1}q^{-\ell}
			\sum_{\mathbf m\ne\mathbf0}|\mathbf m|^{-3\ell}\\
			&\ll_\ell r^{-\ell}
			\qquad(\ell\ge2).
		\end{aligned}
		\tag*{(4.20)}
		\]
		Each cycle of length $2$ contributes $\langle X_{r,M}^2\rangle$.
		Let $L$ be the sum of the half-lengths of all cycles of length at least $4$. Then $L\ge2$,
		and the remaining cycles of length $2$ number $j-L$. By (4.20) and
		$\langle X_{r,M}^2\rangle\ll r^{-1}\log r$, 
		the sum in (4.19c) corresponding to fixed $\mathcal P\ne\mathcal P'$ is at most
		\[
		r^{-L}\left(r^{-1}\log r\right)^{j-L}
		\ll r^{-j}(\log r)^{j-2}.
		\tag*{(4.20a)}
		\]
		Since the cardinality of $\mathfrak P_{2j}$ depends only on $k$, summing over
		the finitely many pairs of pairings in (4.19c) shows that the overcounting error is
		\[
		O_k\!\left(r^{-k/2}(\log r)^{k/2-2}\right),
		\tag*{(4.21)}
		\]
		which can be absorbed into (2.9).
		
	\end{proof}
	
	\section{Proof of Proposition 2.3}
	
	\begin{lemma}\label{lem:phase}
		Let $\mathbf1(n)\equiv1$, and let $*$ denote Dirichlet convolution. Suppose that the arithmetic function $a=\mathbf1*b$ satisfies
		$\sum_{d\ge1}|b(d)|/d<\infty$.
		Then, as $T\to\infty$,
		\[
		\sum_{n\le T}\frac{a(n)}n e\!\left(\frac{2T}n\right)
		=o(\log T).
		\tag*{(5.1)}
		\]
	\end{lemma}
	
	\begin{proof}
		
		We first prove the assertion when $a(n)\equiv1$. Fix $0<\eta<1/2$.
		On the ranges $n\le T^\eta$ and $T^{1-\eta}<n\le T$,
		the combined contribution of the absolute-value estimates is $2\eta\log T+O(1)$.
		
		On the intermediate range, there is a constant $c_\eta>0$ such that
		\[
		\sup_{I\subseteq[N,2N]}
		\left|\sum_{n\in I\cap\mathbb Z}e\!\left(\frac{2T}n\right)\right|
		\ll_\eta NT^{-c_\eta}
		\quad(T^\eta\le N\le T^{1-\eta}).
		\tag*{(5.2)}
		\]
		By Heath-Brown~\cite[Theorem~1]{HB2016}, if a real-valued function $f$ has a continuous $h$th derivative on an interval $J$ of length $N$, with $|f^{(h)}|\asymp_h\lambda$ and $h\ge3$, then
		\[
		\left|\sum_{n\in J\cap\mathbb Z}e(f(n))\right|
		\ll_{h,\epsilon}
		N^{1+\epsilon}
		\left(
		\lambda^{1/(h(h-1))}
		+N^{-1/(h(h-1))}
		+N^{-2/(h(h-1))}\lambda^{-2/(h^2(h-1))}
		\right).
		\tag*{(5.3)}
		\]
		Take $f(x)=2T/x$, so that $\lambda\asymp_h T/N^{h+1}$.
		Choose a fixed even integer $h\ge4$ such that $(h+1)\eta>1$; the choice of even parity is made only to ensure that
		$f^{(h)}$ is positive, so that the sign convention in the cited theorem applies directly. Write
		\[
		\alpha_h=\frac1{h(h-1)},
		\qquad
		\beta_h=\frac2{h^2(h-1)}.
		\]
		For $T^\eta\le N\le T^{1-\eta}$, after dividing (5.3) by $N^{1+\epsilon}$,
		the three terms in parentheses satisfy, respectively,
		\[
		\begin{aligned}
			\lambda^{\alpha_h}
			&\ll_h T^{-((h+1)\eta-1)\alpha_h},\\
			N^{-\alpha_h}
			&\le T^{-\eta\alpha_h},\\
			N^{-2\alpha_h}\lambda^{-2/(h^2(h-1))}
			&\ll_h (N/T)^{\beta_h}
			\le T^{-\eta\beta_h}.
		\end{aligned}
		\tag*{(5.3a)}
		\]
		All three exponents giving decay are strictly positive. Since $N^\epsilon\le T^\epsilon$, first choose
		\[
		0<\epsilon<\frac12\min\left\{
		((h+1)\eta-1)\alpha_h,\,\eta\alpha_h,\,\eta\beta_h
		\right\},
		\]
		and then decrease $c_\eta>0$ if necessary to obtain (5.2) for an interval of length $N$.
		For any subinterval $I\subseteq[N,2N]$, the derivative remains of order
		$T/N^{h+1}$. Applying the same theorem with length $L=|I|\ge1$,
		the total powers of $L$ in the three terms on the right-hand side of (5.3) are, respectively,
		$1+\epsilon$, $1+\epsilon-\alpha_h$, and
		$1+\epsilon-2\alpha_h$, all of which are positive.
		Thus $L$ may be replaced by $N$ to obtain the same upper bound.
		A subinterval of length less than $1$ contains at most two integers. Requiring also that $c_\eta<\eta/2$,
		we have $NT^{-c_\eta}\ge T^{\eta/2}$, so the trivial estimate is also covered by (5.2).
		Hence (5.2) holds uniformly for all subintervals.
		
		A dyadic decomposition of the intermediate range, followed by partial summation, gives
		\[
		\sum_{T^\eta<n\le T^{1-\eta}}
		\frac1n e\!\left(\frac{2T}n\right)
		\ll_\eta T^{-c_\eta}\log T=o(1).
		\tag*{(5.4)}
		\]
		Letting first $T\to\infty$ and then $\eta\to0$, we obtain
		$S(T):=\sum_{n\le T}n^{-1}e(2T/n)=o(\log T)$.
		
		In the general case, convolution gives
		\[
		\sum_{n\le T}\frac{a(n)}n e\!\left(\frac{2T}n\right)
		=\sum_{d\le T}\frac{b(d)}d S(T/d).
		\tag*{(5.5)}
		\]
		Fix $D$. The contribution from $d\le D$ is $o_D(\log T)$,
		whereas the absolute value of the contribution from $d>D$ is at most
		\[
		(1+\log T)\sum_{d>D}\frac{|b(d)|}d.
		\tag*{(5.6)}
		\]
		Letting first $T\to\infty$ and then $D\to\infty$ proves the assertion.
		
	\end{proof}
	
	\begin{proof}[Proof of Proposition~\ref{prop:variance}]
		
		Write
		\[
		h(q)=\prod_{p\mid q}(1-p^{-2})^{-1},
		\qquad
		a_{\mathbf m}(q)=\mu(q)^2h(q)^2
		\mathbf1_{\{\gcd(m_1,m_2,q)=1\}}.
		\tag*{(5.7)}
		\]
		Fix $\mathbf m\ne\mathbf0$ and let $g=\gcd(|m_1|,|m_2|)\ge1$.
		Write $a_{\mathbf m}=\mathbf1*b_{\mathbf m}$.
		For $p\nmid g$,
		\[
		b_{\mathbf m}(p)=h(p)^2-1=O(p^{-2}),
		\quad b_{\mathbf m}(p^2)=-h(p)^2,
		\quad b_{\mathbf m}(p^j)=0\quad(j\ge3).
		\tag*{(5.8)}
		\]
		For $p\mid g$, we have $b_{\mathbf m}(p)=-1$ and
		$b_{\mathbf m}(p^j)=0$ for $j\ge2$.
		Therefore, for every $\alpha>1/2$,
		\[
		\sum_{d\ge1}\frac{|b_{\mathbf m}(d)|}{d^\alpha}<\infty.
		\tag*{(5.9)}
		\]
		Write
		\[
		\begin{aligned}
			\kappa_{\mathbf m}
			&=\sum_{d\ge1}\frac{b_{\mathbf m}(d)}d\\
			&=\prod_{p\mid g}(1-p^{-1})
			\prod_{p\nmid g}\left[
			(1-p^{-1})\left(1+\frac{h(p)^2}p\right)\right]>0.
		\end{aligned}
		\tag*{(5.10)}
		\]
		Fix $\alpha\in(1/2,1)$. By (5.9) and
		$\log(2d)\ll_\alpha d^{1-\alpha}$, we have
		\[
		\begin{aligned}
			\sum_{d\ge1}\frac{|b_{\mathbf m}(d)|\log(2d)}d
			&\ll_\alpha\sum_{d\ge1}\frac{|b_{\mathbf m}(d)|}{d^\alpha}<\infty,\\
			\frac1T\sum_{d\le T}|b_{\mathbf m}(d)|
			&\ll_{\mathbf m,\alpha}T^{\alpha-1},\\
			\sum_{d>T}\frac{|b_{\mathbf m}(d)|}d
			&\ll_{\mathbf m,\alpha}T^{\alpha-1}.
		\end{aligned}
		\tag*{(5.10a)}
		\]
		Write $H_n=\sum_{1\le j\le n}j^{-1}$, and let $\gamma$ denote Euler's constant.
		The estimate $H_n=\log n+\gamma+O(n^{-1})$ implies that
		$H_{\lfloor x\rfloor}=\log x+\gamma+O(x^{-1})$ for $x\ge1$.
		By the convolution identity,
		\[
		\begin{aligned}
			\sum_{q\le T}\frac{a_{\mathbf m}(q)}q
			&=\sum_{d\le T}\frac{b_{\mathbf m}(d)}d H_{\lfloor T/d\rfloor}\\
			&=(\log T+\gamma)\sum_{d\le T}\frac{b_{\mathbf m}(d)}d
			-\sum_{d\le T}\frac{b_{\mathbf m}(d)\log d}d\\
			&\quad+O\!\left(\frac1T\sum_{d\le T}|b_{\mathbf m}(d)|\right).
		\end{aligned}
		\tag*{(5.10b)}
		\]
		Equation (5.10a) shows that the sum involving $\log d$ is bounded.
		Replacing $\sum_{d\le T}b_{\mathbf m}(d)/d$ by $\kappa_{\mathbf m}$ in (5.10b)
		produces an error of $O_{\mathbf m,\alpha}(T^{\alpha-1}\log T)$.
		Taking $\alpha=3/4$, we obtain
		\[
		\sum_{q\le T}\frac{a_{\mathbf m}(q)}q
		=\kappa_{\mathbf m}\log T+O_{\mathbf m}(1).
		\tag*{(5.11)}
		\]
		Lemma~\ref{lem:phase} gives
		\[
		\sum_{q\le T}\frac{a_{\mathbf m}(q)}q
		\sin\!\left(\frac{4\pi T}q\right)
		=o_{\mathbf m}(\log T).
		\tag*{(5.12)}
		\]
		
		Let $T=r|\mathbf m|$. For $q\le T$, using
		\[
		J_1(t)^2=\frac{1-\sin(2t)}{\pi t}+O(t^{-2})
		\qquad(t\ge1),
		\tag*{(5.13)}
		\]
		together with (2.3), we obtain
		\[
		\begin{aligned}
			\sum_{q\ge1}
			\mathbf1_{\{\gcd(m_1,m_2,q)=1\}}
			\left|H\!\left(\frac{\mathbf m}{q}\right)\right|^2
			&=\frac{18}{\pi^6r|\mathbf m|^3}
			\sum_{q\le T}\frac{a_{\mathbf m}(q)}q
			\left(1-\sin\frac{4\pi T}q\right)\\
			&\quad+O\!\left(\frac1{r|\mathbf m|^3}\right).
		\end{aligned}
		\tag*{(5.14)}
		\]
		Here the tail with $q>T$ is estimated using the first bound in (4.2);
		the sum of the Bessel remainder terms for $q\le T$ is also $O(r^{-1}|\mathbf m|^{-3})$.
		Thus, for every fixed $M$,
		\[
		\begin{aligned}
			\langle X_{r,M}^2\rangle&=\sigma_M^2\frac{\log r}r
			+o_M\!\left(\frac{\log r}r\right),\\
			\sigma_M^2&=\frac{18}{\pi^6}
			\sum_{0<|\mathbf m|\le M}
			\frac{\kappa_{\mathbf m}}{|\mathbf m|^3}.
		\end{aligned}
		\tag*{(5.15)}
		\]
		
		Let $S=\sum_{\mathbf m\ne\mathbf0}|\mathbf m|^{-3}$.
		Decomposing the lattice points according to $\gcd(m_1,m_2)$ gives
		\[
		\sum_{\substack{\mathbf m\ne\mathbf0\\
				\gcd(m_1,m_2,q)=1}}|\mathbf m|^{-3}
		=S\prod_{p\mid q}(1-p^{-3}).
		\tag*{(5.16)}
		\]
		On the other hand, since $a_{\mathbf m}=\mathbf1*b_{\mathbf m}$,
		\[
		\frac1x\sum_{q\le x}a_{\mathbf m}(q)
		=\sum_{d\le x}b_{\mathbf m}(d)
		\frac{\lfloor x/d\rfloor}{x}.
		\tag*{(5.16a)}
		\]
		By (5.10a) and $\sum_d|b_{\mathbf m}(d)|/d<\infty$,
		the right-hand side of (5.16a) tends to
		$\sum_db_{\mathbf m}(d)/d=\kappa_{\mathbf m}$. Moreover, since
		$0\le a_{\mathbf m}(q)\le\zeta(2)^2$ and
		$\sum_{\mathbf m\ne\mathbf0}|\mathbf m|^{-3}<\infty$, 
		dominated convergence allows us to interchange the sum over $\mathbf m$ with the limiting average over $q$. Thus (5.16) gives
		\[
		\begin{aligned}
			\sum_{\mathbf m\ne\mathbf0}
			\frac{\kappa_{\mathbf m}}{|\mathbf m|^3}
			&=\lim_{x\to\infty}\frac1x\sum_{q\le x}
			\mu(q)^2h(q)^2
			\sum_{\substack{\mathbf m\ne\mathbf0\\
					\gcd(m_1,m_2,q)=1}}|\mathbf m|^{-3}\\
			&=S\lim_{x\to\infty}\frac1x\sum_{q\le x}c(q),
		\end{aligned}
		\tag*{(5.16b)}
		\]
		where
		\[
		c(q)=\mu(q)^2h(q)^2\prod_{p\mid q}(1-p^{-3}).
		\]
		The function $c$ is multiplicative, and
		\[
		c(p)=h(p)^2(1-p^{-3})=1+O(p^{-2}),
		\qquad c(p^j)=0\quad(j\ge2).
		\]
		Writing $c=\mathbf1*\beta$, we have
		\[
		\beta(p)=c(p)-1=O(p^{-2}),
		\qquad \beta(p^2)=-c(p)=O(1),
		\qquad \beta(p^j)=0\quad(j\ge3).
		\]
		The Euler product therefore gives
		\[
		\sum_{d\ge1}\frac{|\beta(d)|}{d}<\infty.
		\tag*{(5.16c)}
		\]
		The same convolution-averaging argument as in (5.16a) shows that
		\[
		\lim_{x\to\infty}\frac1x\sum_{q\le x}c(q)
		=\sum_{d\ge1}\frac{\beta(d)}d
		=\prod_p\left(1+\frac{\beta(p)}p+
		\frac{\beta(p^2)}{p^2}\right)
		=\prod_p(1-p^{-1})\left(1+\frac{c(p)}p\right).
		\]
		Substituting this into (5.16b), we obtain
		\[
		\begin{aligned}
			\sum_{\mathbf m\ne\mathbf0}
			\frac{\kappa_{\mathbf m}}{|\mathbf m|^3}
			&=S\prod_p\left[(1-p^{-1})
			\left(1+\frac{h(p)^2(1-p^{-3})}p\right)\right]\\
			&=S\prod_p\left(1-\frac1{(p+1)^2}\right).
		\end{aligned}
		\tag*{(5.17)}
		\]
		Finally, the local factors at each prime satisfy
		\[
		(1-p^{-1})
		\left(1+\frac{p^3-1}{(p^2-1)^2}\right)
		=1-\frac1{(p+1)^2}.
		\tag*{(5.18)}
		\]
		It follows that $\sigma_M^2\to\sigma^2$.
		By (4.1) and (5.15),
		\[
		\limsup_{r\to\infty}
		\left|\frac{r\langle X_r^2\rangle}{\log r}-\sigma_M^2\right|
		\ll\frac1M.
		\tag*{(5.19)}
		\]
		Letting $M\to\infty$ now gives (2.10).
		
	\end{proof}
	
	\section{Proof of the Central Limit Theorem}
	
	\begin{proof}[Proof of Theorem~\ref{thm:clt}]
		
		For $r\ge3$ and $M\ge1$, let
		\[
		\begin{aligned}
			Y_r(\mathbf x)&=\sqrt{\frac r{\sigma^2\log r}}X_r(\mathbf x),\\
			Z_{r,M}(\mathbf x)&=\sqrt{\frac r{\sigma^2\log r}}X_{r,M}(\mathbf x).
		\end{aligned}
		\tag*{(6.1)}
		\]
		For fixed $M$, (4.1) and (5.15) give
		\[
		\begin{aligned}
			\limsup_{r\to\infty}\langle|Y_r-Z_{r,M}|^2\rangle
			&\ll\frac1M,\\
			\lim_{r\to\infty}\frac{r\langle X_{r,M}^2\rangle}{\sigma^2\log r}
			&=\frac{\sigma_M^2}{\sigma^2}.
		\end{aligned}
		\tag*{(6.2)}
		\]

		When $1\le M\le\log r$, multiplying (2.9) by $(r/(\sigma^2\log r))^{k/2}$ gives
		\[
		\langle Z_{r,M}^k\rangle
		=
		\begin{cases}
			\displaystyle
			\frac{k!}{2^{k/2}(k/2)!}
			\left(\frac{r\langle X_{r,M}^2\rangle}{\sigma^2\log r}\right)^{k/2}
			+O_k\!\left(\frac{M^2}{(\log r)^2}\right),
			&k\ge4\text{ even},\\[6pt]
			\displaystyle O_k\!\left(\frac{M^{3/2}}{(\log r)^{3/2}}\right),
			&k\ge3\text{ odd}.
		\end{cases}
		\tag*{(6.3)}
		\]
		The first moment is zero, and the limit of the second moment is given by (6.2).
		
		For fixed $r,M$, the function $X_{r,M}$ is a uniform limit of bounded periodic functions.
		For every continuous function $\varphi:\mathbb R\to\mathbb R$,
		the mean $\langle\varphi(Z_{r,M})\rangle$ exists, so $Z_{r,M}$ has a compactly supported spatial limiting distribution.
		The normal distribution is uniquely determined by its moments. Hence, by (6.3) and the method of moments,
		for fixed $M$, these distributions converge as $r\to\infty$ to
		$\mathcal N(0,\sigma_M^2/\sigma^2)$.
		
		For fixed $r$, the function $G(r;\mathbf x)$ takes only finitely many integer values.
		On this finite set, the indicator of each singleton can be represented as a polynomial in $G(r;\mathbf x)$.
		By Proposition~\ref{prop:fourier}, the spatial frequencies of all these values exist, so $Y_r$ also has a spatial limiting distribution.
		
		For each fixed $t\in\mathbb R$, using
		$|e^{ita}-e^{itb}|\le |t||a-b|$, the Cauchy--Schwarz inequality, and (6.2), we obtain
		\[
		\begin{aligned}
			&\limsup_{r\to\infty}
			\left|\langle e^{itY_r}\rangle-e^{-t^2/2}\right|\\
			&\qquad\ll\frac{|t|}{\sqrt M}
			+\left|\exp\!\left(-\frac{\sigma_M^2t^2}{2\sigma^2}\right)
			-e^{-t^2/2}\right|.
		\end{aligned}
		\tag*{(6.4)}
		\]
		Letting $M\to\infty$ and using $\sigma_M^2\to\sigma^2$, we obtain
		$\langle e^{itY_r}\rangle\to e^{-t^2/2}$.
		By L\'evy's continuity theorem, the spatial limiting distributions of $Y_r$ converge weakly to the standard normal distribution.
		Since
		\[
		Y_r(\mathbf x)
		=\frac{G(r;\mathbf x)-6r^2/\pi}
		{\sqrt{\sigma^2r\log r}},
		\tag*{(6.5)}
		\]
		and the standard normal distribution is continuous everywhere, for every $\nu\in\mathbb R$,
		\[
		\begin{aligned}
			&\lim_{r\to\infty}\lim_{R\to\infty}
			\frac1{\pi R^2}
			\operatorname{meas}\left\{
			\mathbf x\in\mathbb R^2:|\mathbf x|\le R,\,
			\frac{G(r;\mathbf x)-6r^2/\pi}
			{\sqrt{\sigma^2r\log r}}\ge\nu
			\right\}\\
			&\qquad=\frac1{\sqrt{2\pi}}
			\int_\nu^\infty e^{-y^2/2}\,dy.
		\end{aligned}
		\tag*{(6.6)}
		\]
		
	\end{proof}
	
	\section{Comparison with Two-Dimensional Poisson Process}
	
	Let $\Pi$ be a homogeneous Poisson point process on $\mathbb R^2$ with intensity $\rho=6/\pi^2$,
	and let $N_r$ denote the number of points in the disk $|\mathbf x|\le r$.
	For every bounded Borel set $A$, the random variable $\Pi(A)$ has a Poisson distribution with parameter $\rho\operatorname{meas}(A)$,
	and the counts in pairwise disjoint sets are independent; see~\cite[Definition~3.1]{LP2017}.
	Therefore,
	\[
	\begin{aligned}
		N_r&\sim\operatorname{Pois}(\lambda_r),
		\qquad \lambda_r=\rho\pi r^2=\frac{6r^2}{\pi},\\
		\mathbb E N_r&=\operatorname{Var}(N_r)=\lambda_r.
	\end{aligned}
	\tag*{(7.1)}
	\]
	For each fixed $t\in\mathbb R$, Taylor expansion gives
	\[
	\begin{aligned}
		\mathbb E\exp\!\left(it\frac{N_r-\lambda_r}{\sqrt{\lambda_r}}\right)
		&=\exp\!\left\{
		\lambda_r\left(e^{it/\sqrt{\lambda_r}}-1-
		\frac{it}{\sqrt{\lambda_r}}\right)\right\}\\
		&\longrightarrow e^{-t^2/2}.
	\end{aligned}
	\tag*{(7.2)}
	\]
	Thus the Poisson count satisfies
	\[
	\frac{N_r-6r^2/\pi}{r\sqrt{6/\pi}}
	\xrightarrow{\mathrm d}\mathcal N(0,1)
	\qquad(r\to\infty).
	\tag*{(7.3)}
	\]
	
	For the coprime pair count, Propositions~\ref{prop:fourier} and~\ref{prop:variance} give
	\[
	\begin{aligned}
		\langle G(r;\cdot)\rangle&=\lambda_r,\\
		\left\langle\bigl(G(r;\cdot)-\lambda_r\bigr)^2\right\rangle
		&=r^2\langle X_r^2\rangle\sim\sigma^2r\log r.
	\end{aligned}
	\tag*{(7.4)}
	\]
	The two counts have the same mean.
	The standard deviation of the Poisson count is $r\sqrt{6/\pi}$, whereas the spatial standard deviation of the coprime pair count is asymptotic to $\sqrt{\sigma^2r\log r}$.
	Under the normalization used for the Poisson count, the coprime pair count converges to zero in spatial mean square:
	\[
	\begin{aligned}
		\left\langle\left(
		\frac{G(r;\cdot)-\lambda_r}{r\sqrt{6/\pi}}
		\right)^2\right\rangle
		&=\frac{\left\langle\bigl(G(r;\cdot)-\lambda_r\bigr)^2\right\rangle}
		{\operatorname{Var}(N_r)}\\
		&\sim\frac{\pi\sigma^2}{6}\frac{\log r}{r}
		\longrightarrow0.
	\end{aligned}
	\tag*{(7.5)}
	\]
	
	The difference between the variances can also be expressed in terms of fluctuations per unit area:
	\[
	\frac{\left\langle\bigl(G(r;\cdot)-\lambda_r\bigr)^2\right\rangle}{\pi r^2}
	\sim\frac{\sigma^2}{\pi}\frac{\log r}{r}\longrightarrow0,
	\qquad
	\frac{\operatorname{Var}(N_r)}{\pi r^2}=\rho.
	\tag*{(7.6)}
	\]
	Intuitively, coprime pairs are more uniformly distributed than the points of a Poisson process.

	\section{Statement on the Use of AI}
	GPT-6 suggested to the author the combinatorial method used in the proof of Proposition 2.2 and provided Lemma 5.1 and its proof. In addition, GPT-6 polished the Chinese draft and translated it into English. The author has reviewed all AI-generated content and takes responsibility for it.

\end{document}